\documentclass [12pt]{article}
\usepackage {a4}
\usepackage {amsfonts}
\usepackage {amssymb,amsmath,amsthm}
\usepackage [utf8]{inputenc}
\usepackage [english]{babel}
\usepackage{enumerate}
\usepackage{graphicx}
\usepackage{fullpage}
\usepackage{bbm}
\usepackage{hyperref}

\newcommand{\R}{\mathbb{R}}

\newcommand{\Z}{\mathbb{Z}}
\newcommand{\Oo}{\mathcal{O}}
\newcommand{\pr}[1]{\mathbb{P}\left( #1 \right)}

\newcommand{\supp}{\mathrm{supp}}

\newtheorem{theorem}{Theorem}[section]

\newtheorem{lemma}[theorem]{Lemma}
\newtheorem{corollary}[theorem]{Corollary}
\newtheorem{proposition}[theorem]{Proposition}

\theoremstyle{definition}
\newtheorem{remark}{Remark}[section]
\newtheorem{assumption}{Assumption}

\author{Michał Kotowski \and B{\'a}lint Vir{\'a}g}
\title{Non-Liouville groups with return probability exponent at most $1/2$}

\begin{document}
\maketitle

\begin{abstract}
We construct a finitely generated group $G$ without the Liouville property such that the return probability of a random walk satisfies $p_{2n}(e,e) \gtrsim e^{-n^{1/2 + o(1)}}$. Recent results suggest that $1/2$ is indeed the smallest possible return probability exponent for non-Liouville groups. Our construction is based on permutational wreath products over tree-like Schreier graphs and the analysis of large deviations of inverted orbits on such graphs. 
\end{abstract}

\section{Introduction}

One of the basic topics of study in probability and group theory is the behavior of random walks on Cayley graphs of finitely generated groups. Among the interesting parameters of a random walk is the {\it return probability} $p_{2n}(e,e)$. There are examples for which it decays polynomially in $n$ (like $\Z^d$ or, more generally, groups of polynomial volume growth) or exponentially (which is the case exactly for nonamenable groups). Other, intermediate types of behavior are also possible, which motivates the study of possible exponents $\gamma$ for which $p_{2n}(e,e) \approx e^{-n^{\gamma}}$. For example, every group of exponential growth must have $\gamma \geq 1/3$ (see \cite{varopoulos}).

Another important parameter is the {\it speed} (or drift) of the random walk. The average distance $\mathbb{E} d(X_0, X_n)$ of the random walk from the origin after $n$ steps may grow linearly with $n$, in which case we say that the random walk has {\it positive speed}, or slower, in which case we say that the random walk has {\it zero speed}. It is thus interesting to ask what exponents $\beta < 1$ such that $\mathbb{E} d(X_0, X_n) \approx n^{\beta}$ are possible. For example, it is known that for every finitely generated group we have $\beta \geq 1/2$ \cite{lee-peres}, but generally computing speed seems more difficult than computing return probabilities. Note that the exponents $\gamma$ and $\beta$ as above need not exist (the return probability and average distance from the origin can oscillate at different scales, see \cite{brieussel}), so in general one should speak about $\liminf$ and $\limsup$ exponents.

Speed of the random walk is closely related to the properties of harmonic functions on groups. Recall that a group has the {\it Liouville property} (with respect to some generating set) if every bounded harmonic function on its Cayley graph is constant. A classical result (see for example discussion in \cite[Chapter 9]{gabor}) says that for groups (though not for general transitive graphs) having positive speed is equivalent to non-Liouville property. Note, however, that it is not known if this property is independent of the generating set (or, more generally, the step distribution of the random walk), which is in contrast to return probabilities, whose decay rate is stable under quasi-isometries (\cite{pittet-saloff-coste}).

The motivation for this paper is the following remarkable theorem (which is a corollary of a more general result from \cite{laurent-tianyi}): if the return probability satisfies $p_{2n}(e,e) \geq K e^{-cn^{\gamma}}$ for $\gamma < 1/2$ (and some constants $K, c >0$), then the group has the Liouville property \footnote{This theorem was first announced in \cite{gournay}, but the proof there relies on an assumption about off-diagonal heat kernel bounds which has not been proved to hold except for groups of polynomial growth.}. In particular, it has zero speed for every generating set (since, as mentioned above, the property $\gamma < 1/2$ is invariant under quasi-isometries). This is the first known general result connecting return probabilities with speed and showing quasi-isometry invariance of the Liouville property for a broad class of groups. For more discussion of possible relationships between these exponents (and also other quantities like entropy or volume growth) and numerous examples, see (\cite[Section 4]{gournay}).

This result does not characterize the Liouville property, since there exist groups with $\gamma$ arbitrarily close to $1$ which are still Liouville \cite{erschler-bartholdi-intermediate}. In the other direction, it is natural to ask whether the value $1/2$ in the theorem cited above can be improved, i.e. whether there exist groups with $\gamma$ arbitrarily close to $1/2$ from above (or even equal to $1/2$) which are non-Liouville. Several examples of groups with $\gamma = 1/2$ are known (\cite{saloff-coste}), but they all have the Liouville property.

The main result of our paper is the construction of a finitely generated group which has $\gamma \leq 1/2$, but at the same time is non-Liouville. More precisely, consider the {\it upper return probability exponent}:
\[
\overline{\gamma} = \limsup\limits_{n \to \infty} \frac{ \log |\log p_{2n}(e,e)|}{\log n}
\]
We will prove the following theorem:

\begin{theorem}\label{th:main}
There exists a finitely generated group $G$ and a symmetric finitely supported random walk $\mu$ on $G$ such that $G$ is non-Liouville with respect to $\mu$ and the upper return probability exponent satisfies $\overline{\gamma} \leq 1/2$.
\end{theorem}

In other words, the return probability for this random walk satisfies the lower bound $p_{2n}(e,e) \geq K e^{-n^{1/2 + o(1)}}$ for some constant $K>0$ and the random walk has positive speed. Previously the smallest known return probability exponent for a non-Liouville group was $3/5$ for the lamplighter group $\Z_2 \wr \Z^3$ (\cite{saloff-coste}). Determining a good upper bound for the return probability on $G$ seems to be an interesting problem in its own right.

\subsection*{Idea of the construction}

We now sketch the idea of our construction. Among the groups for which one can provide precise asymptotics for the return probabilities are the lamplighter groups $\Z_2 \wr \Z^d$. It is known \cite[Theorem 3.5]{saloff-coste} that in this case we have $\gamma = \frac{d}{d+2}$ - in particular, for $d=2$ we obtain a group with $\gamma = 1/2$. The group $\Z_2 \wr \Z^2$ is Liouville, but only barely so, as its speed satisfies $\mathbb{E} d(X_0, X_n) \approx \frac{n}{\log n}$. Thus the idea is that if one could in some sense do the lamplighter construction for $d \approx 2 + \varepsilon$ for some small $\varepsilon$, or even $d \approx 2 + o(1)$ (which would correspond to putting the lamps on a graph with volume growth slightly faster than quadratic), one would get a group with $\gamma$ close to $1/2$ and, if the graph grows quickly enough, positive speed.

The problem is of course that there are no ``$2 + \varepsilon$''-dimensional Cayley graphs. Nevertheless, one can carry out the lamplighter construction over an almost two dimensional graph (this time only a Schreier graph, not a Cayley graph) if we move from ordinary wreath products to {\it permutational wreath products}. They are a generalization of wreath products to the setting where a finitely generated group acts on a Schreier graph (the usual wreath product would correspond to the group acting on itself). They share some similarities with the ordinary lamplighter groups, but there are also important differences (see Section \ref{section:preliminaries} for more discussion).

For the construction of the group $G$ we define a tree-like Schreier graph $S$ which grows sufficiently quickly so that the simple random walk on it is transient. The graph naturally defines a group $\Gamma$ which we call the {\it bubble group}. The group $G$ is then defined as the permutational wreath product $\Z_2 \wr_{S} \Gamma$, which corresponds to putting $\Z_2$-valued lamps on $S$, with $\Gamma$ acting on lamp configurations. One can show that this product is non-Liouville as soon as $S$ is transient.

In the case of the usual lamplighter group $\Z_2 \wr \Z^d$, providing a lower bound on the return probability requires understanding the range of the simple random walk on the underlying base graph $\Z^d$ (roughly speaking, the dominant contribution to returning to identity in the wreath product comes from switching off all the lamps visited, and the number of visited lamps is governed by the range of the underlying random walk). To obtain a sharp bound we need to know certain large deviation estimates for the range, not only its average size. For permutational wreath products the situation is more complicated, as the size of the lamp configuration on $S$ is governed not by the range of the simple random walk on $S$, but by the {\it inverted orbit process}. This is a different random process which is generally not as well understood. In our case the graph $S$ has large parts which locally look like $\Z$, so one can still analyze the inverted orbits using large deviation estimates for $\Z$.

As a closing remark we mention that the idea of using ``bubble graphs'' comes from looking at orbital Schreier graphs of certain groups of bounded activity acting on trees (used in \cite{balint-gidi} to provide examples of groups with speed exponents between $3/4$ and $1$), which have somewhat similar branching structure. In particular, Gady Kozma (personal communication, see also \cite{gady-gidi}) proposed looking at similar groups permuting vertices of slowly growing trees as examples in group theory. In general it would be desirable to obtain a better understanding of inverted orbits and probabilistic parameters (return probabilities, speed, entropy) on related groups of this type. Some results along these lines can be found for example in \cite{brieussel}, where entropy and return probability exponents on groups of directed automorphisms of bounded degree trees are analyzed.

\subsection*{Structure of the paper and notation}

The paper is structured as follows. In Section \ref{section:preliminaries} we provide the background on permutational wreath products, inverted orbits and switch-walk-switch random walks used for the wreath products. In Section \ref{section:construction} we define the family of Schreier graphs and bubble groups used in the main construction. In Section \ref{section:bounds} we provide estimates on the size of inverted orbits for random walks on the graph. In Section \ref{section:liouville} we state the theorem used to deduce the non-Liouville property from transience and provide a criterion for checking that the graph defined in the previous section is transient. In Section \ref{section:lower-bound} we fix the Schreier graph and the bubble group, prove the graph's transience and provide lower bounds on return probabilities (using results from Section \ref{section:construction}), thus proving Theorem \ref{th:main}.

Throughout the paper by $c$ we will denote a positive constant (independent of parameters like $m$ or $n$) whose exact value is not important and may change from line to line. We will also use the notation $f(n) \lesssim g(n)$ meaning $f(n) \leq C g(n)$ for some constant $C > 0$.

\section{Preliminaries}\label{section:preliminaries}

Let us recall the notion of a permutational wreath product. Suppose we have a finitely generated group $\Gamma$ acting on a set $S$ and a finitely generated group $\Lambda$ (in our case this group will be finite). For $x \in S$ we will denote the action of $g \in \Gamma$ on $x$ by $x.g$. The graph will usually have a distinguished vertex $o$ called the {\it root}.

The {\it permutational wreath product} $\Lambda \wr_{S} \Gamma$ is the semidirect product $ \bigoplus_{S} \Lambda  \rtimes \Gamma$, where $\Gamma$ acts on the direct sum by permuting the coordinates according to the group action. Elements of the permutational wreath product can be written as pairs $(f, g)$, where $g \in \Gamma$ and $f : S \to \Lambda$ is a function with only finitely many non-identity values. For two such pairs $(f,g)$, $(f', g')$ the multiplication rule is given by:
\[
(f,g)(f', g') = (f f^{' g^{-1}}, g g')
\]
where $f^{ g^{-1}}$ is defined as $f^{ g^{-1}}(x) = f(x.g)$. If $\Gamma$ and $\Lambda$ are finitely generated, then $\Lambda \wr_{S} \Gamma$ is also finitely generated.

By $\supp f$ we will denote the set of vertices of $S$ at which $f(s)$ is not identity.

The usual wreath products (with $S = \Gamma$) are often called lamplighter groups - we think of $f$ as being a configuration of lamps on $S$ and $g$ being the position of a lamplighter. A random walk on the lamplighter group corresponds to the lamplighter doing a random walk on $\Gamma$ and changing values of the lamps along his trajectory.

By analogy with the usual wreath product we will call $\Gamma$ the {\it base group} and $\Lambda$ the {\it lamp group}. There are however important differences in how random walks on permutational wreath products behave. To see this, consider a symmetric probability distribution $\mu$ on $\Gamma$ and a {\it switch-walk-switch} random walk $\tilde{X}_n$ on $\Lambda \wr_{S} \Gamma$:
\[
\tilde{X}_n = \prod\limits_{i=1}^{n} (l_{i},id_{\Gamma}) (id_{\Lambda},g_{i}) (l_{i}',id_{\Gamma})
\]
Here $g_{i}$ are elements of $\Gamma$ chosen independently according to $\mu$ and $l_{i}$, $l_{i}'$ are independent random {\it switches} of the form:
\[
l_{i}(x) = \begin{cases} id_{\Lambda} &\mbox{if } x \neq o \\ 
L & \mbox{if } x = o \end{cases} 
\]
where $L$ is chosen randomly from a fixed symmetric probability distribution on $\Lambda$. We can write $\tilde{X}_{n} = (X_n, Z_n)$, where $Z_n = g_1 \ldots g_n$ is the random walk on $\Gamma$ corresponding to $\mu$ and $X_n$ is a random configuration of lamps on $S$. We will always assume that the probability distribution on $\Lambda$ is nontrivial.

Now observe that if we interpret this walk as a lamplighter walking on $\Gamma$ and switching lamps on $S$, the switches happen at locations $o, o.g_{1}^{-1}, o.g_{2}^{-1}g_{1}^{-1}, \ldots$, $o.g_{n}^{-1}\ldots g_{2}^{-1} g_{1}^{-1}$. For ordinary wreath products, with $o$ being the identity of the base group, this is the same as the orbit of the left Cayley graph, $o, g_{1}^{-1}.o, g_{2}^{-1}g_{1}^{-1}.o, \ldots$, $g_{n}^{-1}\ldots g_{2}^{-1} g_{1}^{-1}.o$. However, in general the set of locations at which switches happen behaves differently from the usual orbit - for example, it does not even have to be connected.

This phenomenon motivates the definition of the inverted orbit. Suppose that, as above, we have a group $\Gamma$, acting from the right on a set $S$, and a word $w = g_1 \ldots g_n$, where $g_{i}$ are generators of $\Gamma$. Given $o \in S$, its {\it inverted orbit} under the word $w$ is the set $\Oo(w)$ = $\{o, o.g_{1}^{-1}, o.g_{2}^{-1} g_{1}^{-1}, \ldots, o.g_{n}^{-1}g_{n-1}^{-1} \ldots g_{1}^{-1} \}$.

Likewise, suppose we have a symmetric probability distribution $\mu$ on $\Gamma$ and the corresponding random walk $Z_{n} = g_1 g_2 \ldots g_n$, where each $g_i \in \Gamma$ is chosen independently according to $\mu$. Given $o \in S$, its inverted orbit under the random walk $Z_n$ is the (random) set $\Oo(Z_n)$ = $\{o, o.g_{1}^{-1}, o.g_{2}^{-1} g_{1}^{-1}, \ldots, o.g_{n}^{-1}g_{n-1}^{-1} \ldots g_{1}^{-1} \}$. We call the set-valued process $\Oo(Z_n)$ the {\it inverted orbit process} on $S$. Abusing the notation slightly we will denote by $Z_n$ both the trajectory of the random walk up to time $n$ and the corresponding group element.

As noted above, this is not the same as the ordinary orbit, which would correspond to the set $\{ o, o.g_{1}, o.g_{1} g_{2}, \ldots, o.g_{1} g_{2} \ldots g_{n} \}$. In particular, the inverted orbit process is not a reversible Markov process.

There are many examples in which permutational wreath products behave differently from the usual wreath products. For instance, while usual wreath products always have exponential growth if the base group is infinite and the lamp group is nontrivial, permutational wreath products can have intermediate growth. This is directly related to the difference between the behavior of inverted orbits and ordinary orbits (see \cite{erschler-bartholdi-growth} and other work by Bartholdi and Erschler).

\section{The bubble group}\label{section:construction}

We start by defining the Schreier graph and the group acting on it.  Fix a {\it scaling sequence} $1 \leq \alpha_1 \leq \alpha_2 \leq ...$. The corresponding graph $S(\alpha)$ is constructed as follows. The edges of the graph are labelled by two generators $a$, $b$ and their inverses. The graph is constructed recursively - the first level consists of the {\it root} $o$, followed by a cycle of length $2 \alpha_1$. The $n$-th level is defined in the following way -  place a cycle of length $3$ (called a {\it branching cycle}), labelled cyclically by $b$, in the middle of each cycle from the previous level so that each cycle is split into two paths. Then each of the remaining two vertices on the branching cycle is followed by a cycle of length $2 \alpha_n$ (see the picture below). For a given cycle from the $n$-th level we will denote its starting point by $b_n$ (with $b_1 = o$). We will think of the graph as extending to the right, so the particles most distant from the root are the rightmost ones.

The edges of every path are labelled by $a$ and $a^{-1}$ and every vertex, apart from the vertices on the branching cycles, is mapped by $b$ and $b^{-1}$ to itself.

\begin{figure}[h!]
  \centering
    \includegraphics[scale=0.27]{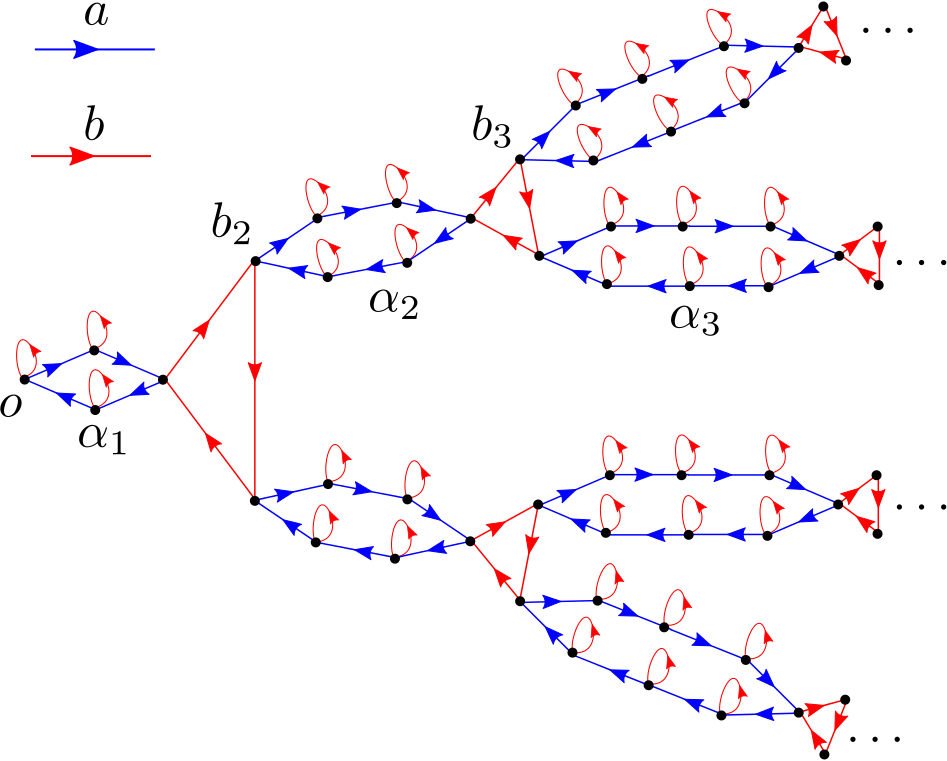}
\caption{First three levels of the Schreier graph $S(\alpha)$ for $\alpha_1 = 2$, $\alpha_2 = 3$, $\alpha_3 = 4$.}
\end{figure}

From this graph we obtain a group in natural way. Each of the generators $a,b$ and their inverses defines a permutation of the vertices of $S(\alpha)$ and we define {\it the bubble group} $\Gamma(\alpha)$ as the group generated by $a$ and $b$. $\Gamma(\alpha)$ acts on $S(\alpha)$ from the right and by $x.g$ we will denote the action of $g \in \Gamma(\alpha)$ on a vertex $x \in S(\alpha)$. By $d(x,y)$ we will denote the distance of $x$ and $y$ in $S$.

\section{Bounds on the inverted orbits}\label{section:bounds}

In what follows we denote $S(\alpha)$ and $\Gamma(\alpha)$ by $S$ and $\Gamma$ for simplicity.

Consider the simple random walk $Z_n$ on $\Gamma$ (each of the generators $a, b, a^{-1}, b^{-1}$ is chosen with equal probability) and the corresponding inverted orbit process $\Oo(Z_n)$ on $S$. Our goal is to prove that, for a suitably chosen scaling sequence, the inverted orbit process on the Schreier graph $S$ satisfies the same bound on the range as the simple random walk on $\Z$.

Let $s_k = \alpha_1 + \ldots + \alpha_k + k$ be the total distance from $o$ to the branching point $b_{k+1}$, with $s_{0} = 0$.
\begin{assumption}\label{eq:scaling-assumption}
From now on we will assume that the scaling sequence satisfies:
\[
d s_{k-1} \leq \alpha_{k}
\]
for all $k \geq 2$ and some constant $d > 0$.
\end{assumption}

In other words, we require each level to be of length comparable to the sum of all previous levels, so that the graph $S$ is like a tree with branches of length growing at least exponentially.

We want to reduce bounding the inverted orbit of $Z_{n}$ to analyzing a one-dimensional random walk. To any given word $w$ in $a,b,a^{-1}, b^{-1}$ we can naturally associate a path on $\Z$  - $a$ corresponds to moving right, $a^{-1}$ corresponds to moving left and $b, b^{-1}$ both correspond to staying put. As $a,b,a^{-1}, b^{-1}$ appear with equal probability as steps of $Z_n$, we get that the random walk $Z_{n} = g_{1} \ldots g_{n}$ projects to a {\it lazy random walk} $\hat{Z}_{n} = \hat{g}_{1} \ldots \hat{g}_{n}$ on $\Z$ (started at the origin), which moves right with probability $1/4$, moves left with probability $1/4$ and stays put with probability $1/2$.

Let $R_{n}$ denote the range of $\hat{Z}_{n}$, i.e. the set of all vertices visited by $\hat{Z}_{n}$ up to time $n$. Let $A_{n,m}$ denote the event that the range of $\hat{Z}_{n}$ is contained in a small ball, $A_{n,m} = \{ R_{n} \subseteq [-m, m] \}$. We have the following lemma on large deviations of a lazy random walk:

\begin{lemma}\label{lm:lazy}
For every $n$, $m \geq 1$ we have:
\[
\pr{A_{n,m}} = \pr{ R_{n} \subseteq [-m, m]} \gtrsim e^{-c \frac{n}{m^2}}
\]
\end{lemma}

\begin{proof}
See \cite[Lemma 1.2]{alexopoulos} (or \cite[Theorem 3.12]{saloff-coste} for a more general case).
\end{proof}

The following simple observation will be useful: if the trajectory $\hat{g}_{1}\ldots \hat{g}_{n}$ has its range bounded between $-m$ and $m$, then for any subword $w = g_{k} g_{k+1} \ldots g_{l}$ the trajectory $\hat{g}_{k} \hat{g}_{k+1} \ldots \hat{g}_{l}$ (started at the origin) has its range bounded between $-2m$ and $2m$. Furthermore $w$ has range bounded between $-2m$ and $2m$ if and only if $w^{-1} = g_{l}^{-1} \ldots g_{k+1}^{-1} g_{k}^{-1}$ satisfies the same bound.

Now consider a particle moving on the graph according to the action of a word $w$ or its inverse, starting at some vertex $x$.  For two vertices $y,z$ we will say that $y$ is \emph{to the right} (resp. \emph{to the left}) of $z$ if $d(o,z) < d(o,y)$ (resp. $d(o,z) > d(o,y)$).

We will repeatedly use the following lemma (which is a direct consequence of the observation above and the assumption $A_{n,m}$):

\begin{lemma}\label{lm:all-x}
Suppose that $A_{n,m}$ holds for a word $w$. Let $v$ be a vertex visited by the particle at some sequence of times and consider any subword $w'$ of $w$ corresponding to the minimal part of the trajectory between two subsequent visits to $v$ (or after the last visit, if $v$ is not visited after certain time). Whenever the particle visits $v$, if there is no branching cycle within distance $2m$ to the right (resp. to the left) of $v$, then $w'$ will move the particle no further away than $2m$ to the right (resp. to the left) from $v$.
\end{lemma}

\begin{theorem}\label{th:all-x}
Suppose that the scaling sequence satisfies Assumption \ref{eq:scaling-assumption}. If $A_{n,m}$ holds for the trajectory $Z_{n} = g_1 \dots g_n$, then for each $x \in S$ and every subword $w = g_{k} g_{k+1} \ldots g_{l}$ or its inverse we have $d(x, x.w) \leq Km$ (for some $K \geq 1)$.
\end{theorem}

\begin{proof}

The idea of the proof is that due to the assumption on exponential-like growth, the largest level contained in $B_{m}(x)$ is roughly of the same size as the whole ball, so we can bound the particle's position by looking only at its behavior at the last level (or levels of comparable size), where it behaves like a walk on $\Z$.

We consider three types of vertices: such that $B_{2m}(x)$ intersects only one level, intersects two levels or intersects at least three levels.
\\ \\
(1) In the first case there is no branching cycle within distance $2m$ from $x$, so the ball $B_{2m}(x)$ is isomorphic to a ball in $\Z$ and we can directly use the assumption $A_{n,m}$ to conclude that the particle stays within distance at most $2m$ from $x$.
\\ \\
(2) In the second case, assume that $x$ belongs to the $k$-th level and the ball intersects also the $k+1$-st level (the case when the ball intersects the $k-1$-st level is analogous). To the left the ball doesn't intersect any branching cycle, so we can again directly use the property $A_{n,m}$. To the right, either the particle doesn't hit any $b_{k+1}$, in which case it is within distance $2m$ to the right of $x$, or it hits $b_{k+1}$ (for one of the two cycles from the $k+1$-st level) - then we can apply Lemma \ref{lm:all-x} with $v = b_{k+1}$ to conclude that it never goes further than $2m$ to the right of $b_{k+1}$. This implies that we always stay within distance at most $4m$ from $x$.
\\ \\
(3) In the third case $x$ must be close to the origin. Namely, if $x$ belongs to the $k$-th level, then at least one of $\alpha_{k-1}, \alpha_{k}, \alpha_{k+1}$ is smaller than $4m$ (since $B_{2m}(x)$ intersects at least three levels). Since $\alpha_{k-1} \leq \alpha_{k} \leq \alpha_{k+1}$, we have $\alpha_{k-1} < 4m$. As $\alpha_{k-1} \geq d s_{k-2}$ by Assumption \ref{eq:scaling-assumption}, we have $(1+d) \alpha_{k-1} \geq d (s_{k-1} - 1)$. Now $B_{2m}(x)$ intersects the $k-1$-st level (otherwise we would have $2m \leq \alpha_{k} \leq \alpha_{k+1}$ and the ball would intersect only two levels), so $d(o,x) \leq s_{k-1} + 2m$. This gives us:
\[
d(o,x) \leq \frac{1+d}{d} \alpha_{k-1} + 1 + 2m \leq \left(2 + \frac{4(1+d)}{d}\right)m + 1 \leq \left(3 + \frac{4(1+d)}{d}\right)m
\]
Thus $x$ belongs to a ball $B_{c_{1} m}(o)$, where $c_{1}$ is the constant on the right hand side of the inequality above.

Now take the first level $l$ which has $\alpha_{l} \geq 4m$. Then $b_{l}$ is to the right of $x$ and $\alpha_{l-1} < 4m$. We have $d(o, b_{l}) =  s_{l-1}$ and $d s_{l-2} \leq \alpha_{l-1}$, so $d (s_{l-1} - 1) \leq (1+d) \alpha_{l-1} < (1+d)4m$. Thus:
\[
d(o, b_{l}) < \frac{4(1+d)}{d}m +1 \leq \left(\frac{4(1+d)}{d} + 1\right)m
\]
Let $c_{2}$ be the constant multiplying $m$ in the inequality above. If the particle stays to the left of $b_{l}$, it is within distance at most $c_{2}m$ from the origin and thus within distance at most $(c_{1} + c_{2})m$ from $x$. If it hits $b_{l}$ at some point, then, as $\alpha_{l} \geq 4m$, for each visit we can apply Lemma \ref{lm:all-x} with $v = b_{l}$ to conclude that the particle stays within distance $4m$ to the right from $b_{l}$, so it is within distance $(4 + c_{2})m$ from the origin and thus within distance $(4 + c_{1} + c_{2})m$ from $x$.

Thus the theorem holds with $K = 4 + c_{1} + c_{2}$.
\end{proof}

\begin{corollary}\label{th:estimate}
Under the assumption of the previous theorem, if $A_{n,m}$ holds, the inverted orbit process $\Oo(Z_n)$ on $S$ satisfies $\Oo(Z_n) \subseteq B_{Km}(o)$, where $B_{Km}(o)$ denotes the ball of radius $Km$ and center $o$ in $S$ (with $K$ as in the previous theorem).
\end{corollary}

\begin{proof}
Recall that $\Oo(Z_n) = \{o, o.g_{1}^{-1}, o.g_{2}^{-1} g_{1}^{-1}, \ldots, o.g_{n}^{-1} g_{n-1}^{-1} \ldots g_{1}^{-1} \}$. We can apply the previous theorem to words of the form $g_{k}^{-1} \ldots g_{2}^{-1} g_{1}^{-1}$ for $k=1, \ldots ,n$. We get that \\ $d(o, o.g_{k}^{-1} \ldots g_{2}^{-1} g_{1}^{-1}) \leq Km$, which proves $\Oo(Z_n) \subseteq B_{Km}(o)$.
\end{proof}

Thus with probability at least a constant times $e^{-c \frac{n}{m^2}}$ no vertex is moved by $Z_{n}$ further than $Km$ from itself and the inverted orbit of $o$ is small (contained in a ball of radius $Km$ around $o$).

\section{Liouville property and transience}\label{section:liouville}

We briefly recall the notions related to the Liouville property and harmonic functions. Given a measure $\mu$ on a group $G$, a function $f : G \to \R$ is said to be {\it harmonic} (with respect to $\mu$) if we have $f(g) = \sum_{h \in G} f(gh)\mu(h)$. $G$ is said to have the {\it Liouville property} if every bounded harmonic function on $G$ is constant. As mentioned in the introduction, this is equivalent to the random walk associated to $\mu$ having zero asymptotic speed. This property a priori depends on the choice of $\mu$ (in the case when $\mu$ is a simple random walk - on the choice of the generating set of $G$).

We want to construct a group which is non-Liouville, i.e. supports nonconstant bounded harmonic functions. For permutational wreath products one can ensure this by requiring that the Schreier graph used in the wreath product is transient:

\begin{theorem}\label{th:liouville}
Let $\Gamma$ and $F$ be nontrivial finitely generated groups and let $\mu$ be a finitely supported symmetric measure on $\Gamma$ whose support generates the whole group. Let $\tilde{\mu}$ be the measure associated to the corresponding switch-walk-switch random walk on the permutational wreath product $F \wr_{S} \Gamma$. If the induced random walk on $S$ is transient, then the group $F \wr_{S} \Gamma$ has nontrivial Poisson boundary, i.e. supports nonconstant bounded harmonic functions (with respect to $\tilde{\mu}$).
\end{theorem}

Related results appear in several places \cite{balint-gidi-2}. The formulation we use here comes from \cite[Proposition 3.5]{erschler-bartholdi}. We briefly sketch the idea of the construction here.

To construct a nonconstant harmonic function on the group, consider the state of the lamp at $o$. Since the walk on $S$ is transient, with probability $1$ this vertex will be visited only finitely many times, so after a certain point the value of the lamp will not change anymore and thus the eventual state $L$ of this lamp is well-defined as $n \to \infty$. Now one can show that for any vertex $x$ the mapping $x \mapsto \mathbb{P}_{x}(L=e)$ (where $\mathbb{P}_{x}$ denotes the probability with respect to a random walk started at $x$) defines a nonconstant bounded harmonic function on the group.

A useful criterion for establishing transience is based on electrical flows (we formulate it for simple random walks). Given a graph $S$, a {\it flow} $I$ from a vertex $o$ is a nonnegative real function on the set of directed edges of $S$ which satisfies Kirchhoff's law: for each vertex except $o$ the sum of incoming values of $I$ is equal to the sum of outgoing values. A {\it unit flow} is a flow for which the outgoing values from $o$ sum up to $1$. The {\it energy} of the flow is given by $\mathcal{E}(I) = \frac{1}{2} \sum\limits_{e} I(e)^2$, where the sum is over the set of all directed edges.

\begin{proposition}[{\cite[Theorem 2.11]{peres-lyons}}]\label{prop:flow}
If a graph $S$ admits a unit flow with finite energy, then $S$ is transient.
\end{proposition}

\section{Lower bound on return probability}\label{section:lower-bound}

Consider the Schreier graph $S(\alpha)$ and the bubble group $\Gamma(\alpha)$, depending on a scaling sequence $\alpha = (\alpha_1, \alpha_2, \ldots)$, as described in Section \ref{section:construction}. As mentioned in the introduction, we would like the graph $S(\alpha)$ to be transient and have ``$2 + o(1)$''-dimensional volume growth, and also satisfy the Assumption \ref{eq:scaling-assumption} on exponential-like growth.

To analyze volume growth, consider $n$ such that $s_{k-1} \leq n < s_{k}$ (following the notation of Section \ref{section:bounds}). Because of the branching structure of $S(\alpha)$,  the size of the ball $B_{n}(o)$ of radius $n$ around $o$ satisfies:
\[
|B_{n}(o)| \leq 2 \left(\alpha_1 + 1 + 2(\alpha_2 + 1) + \ldots + 2^{k-1} (\alpha_{k} + 1) \right)
\]
For a scaling sequence satisfying $\alpha_{k} = \alpha^{k + o(k)}$, with $\alpha > 1$, it is easy to see that the volume of the ball will satisfy:
\begin{equation*}
|B_{n}(o)| \leq n^{1 + \frac{\log 2}{\log \alpha} + o(1)}
\end{equation*}
as $n \to \infty$. In particular if we take $\alpha_{k} = \frac{2^{k}}{f(k)}$ for some positive and sufficiently slowly increasing function $f(k)$, then:
\begin{equation}\label{eq:growth}
|B_{n}(o)| \leq n^{2 + \varepsilon(n)}
\end{equation}
for some nonnegative function $\varepsilon(n) \to 0$ as $n \to \infty$. How slowly $f(k)$ should grow will be determined by the transience requirement.

Consider the graph $S(\alpha)$ and the group $\Gamma(\alpha)$ defined by taking a scaling sequence $\alpha_{k}$ satisfying:
\[
\sum\limits_{k=1}^{\infty} \frac{\alpha_{k}}{2^k} < \infty
\]

\begin{proposition}\label{proposition:transience}
For $\alpha_{k}$ as above the graph $S(\alpha)$ is transient.
\end{proposition}

\begin{proof}
We use the flow criterion from Proposition \ref{prop:flow}. Consider any cycle on the $k$-th level of the graph. If the edge $e$ is on the upper half of the cycle and is labelled by $a$, or is on the lower half of the cycle and is labelled by $a^{-1}$, we take the value of $I(e)$ to be $1/2^{k}$. The two edges labelled by $b$ and $b^{-1}$ adjacent to the rightmost point of the cycle also get the value $1/2^k$ and all other edges have values $0$. One readily checks that this function satisfies Kirchoff's law and its energy is given by:
\[
\mathcal{E}(I) = \frac{1}{2} \sum\limits_{e} I(e)^2 = \sum\limits_{k=1}^{\infty} 2^{k-1} \alpha_{1} \left( \frac{1}{2^k} \right)^2 = \frac{1}{2} \sum\limits_{k=1}^{\infty} \frac{\alpha_{k}}{2^k} 
\]
which is finite by the assumption on the scaling sequence.
\end{proof}

An example of a scaling sequence satisfying this assumption is $\alpha_{k} = \lceil \frac{2^k}{k^2} \rceil$ and from now on we denote by $S$ and $\Gamma$ the graph and the group corresponding to this choice of $\alpha$. One can easily check (by induction) that this scaling sequence satisfies Assumption \ref{eq:scaling-assumption} on exponential-like growth.

The graph $S$ satisfies the volume growth condition $|B_{n}(o)| \leq n^{2 + \varepsilon(n)}$ described above for $\varepsilon(n) \lesssim \frac{\log \log n}{\log n}$ (so that $|B_{n}(o)| \approx n^{2} \log^\delta n$ for some $\delta>0$). We will use the graph $S$ and the group $\Gamma$ to construct a group with the desired behavior of return probabilities.

Consider the permutational wreath product $G = \Z_2 \wr_{S} \Gamma$. Let $Z_n$ be the simple random walk on $\Gamma$ and denote by $\tilde{X}_n = (X_n, Z_n)$ the associated switch-walk-switch random walk on $G$ (with the uniform distribution on the lamp group $\Z_2$).

Denote by $p_{n}(g,h)$ the probability that $\tilde{X}_{n} = h$ given $\tilde{X}_{0} = g$, where $g, h \in G$. To bound the return probability $p_{2n}(e,e)$, for any finite set $A \subseteq G$ we can write, using the symmetry of the random walk and Cauchy-Schwarz inequality:
\[
p_{2n}(e,e) = \sum\limits_{g \in G} p_{n}(e,g) p_{n}(g,e) = \sum\limits_{g \in G} p_{n}(e,g)^2 \geq \sum\limits_{g \in A} p_{n}(e,g)^2 \geq \frac{p_{n}(A)^2}{|A|}
\]
where $p_{n}(A) = \sum\limits_{g \in A} p_{n}(e,g)$ is the probability that $\tilde{X}_{n}$ is in the set $A$ after $n$ steps.

For the usual lamplighter $\Z_{2} \wr \Z^d$ we would take $A$ to be the set of all elements with lamp configurations contained in a ball of radius $n^{\alpha}$ (with $\alpha$ to be optimized later) and lower bound $p_{n}(A)$ by the probability of the simple random walk on $\Z^d$ to be actually confined to a ball of radius $n^{\alpha}$. Since the base group has polynomial growth, the main contribution to $|A|$ comes from the number of lamp configurations, which is of the order of $e^{n^{d \alpha}}$ (as balls in $\Z^d$ have volume growth $\approx n^d$). The probability that the range of a simple random walk on $\Z^d$ is contained in a ball of radius $n^{\alpha}$ can be shown to be of the order of $e^{-n^{1 - 2\alpha}}$. We want these two terms to be of the same order - optimizing for $\alpha$ gives that one should consider balls of radius $n^{\frac{1}{d+2}}$, which gives the correct return probability exponent of $\frac{d}{d+2}$.

We use the same approach for the permutational wreath product $\Z_2 \wr_{S} \Gamma$, the difference being that we are dealing with inverted orbits instead of ordinary random walks and we have to be more careful with estimating the possible positions of the random walker on the base group.

Let $B_{Km}(o)$ be a ball of radius $Km$ around $o$ in $S$ (with $K$ as in Theorem \ref{th:all-x} and $m$ to be optimized later). We will say that a word $w$ has {\it small inverted orbits} if $\Oo(w) \subseteq B_{Km}(o)$. Consider the set $C$ of group elements with the following property: each element of $C$ can be represented by a word $w$ of length $n$ such that $w$ has small inverted orbits and $d(x, x.w) \leq Km$ for every $x \in S$.

Following the same approach as for the ordinary lamplighter group, in the bound above we take $A = \{ (f, \gamma) \in G \, | \,  \supp f \subseteq B_{Km}(o), \, \gamma \in C \}$.

We have to provide a lower bound on $p_{n}(A)$ and an upper bound on the size of $A$.

\begin{theorem}\label{th:ret-prob}
$p_{n}(A) \gtrsim e^{-c \frac{n}{m^2}}$ for all $n \geq 1$.
\end{theorem}

\begin{proof}

We have $p_{n}(A) = \pr{\tilde{X}_n \in A} = \pr{\Oo (Z_n) \subseteq B_{Km}(o), Z_n \in C}$. By Lemma \ref{lm:lazy}, Theorem \ref{th:all-x} and Corollary \ref{th:estimate} with probability at least $ e^{-c \frac{n}{m^2}}$ (up to a multiplicative constant) the random element $Z_n$ simultaneously has small inverted orbits, so $\Oo (Z_n) \subseteq B_{Km}(o)$, and does not move any vertex further than $Km$ from itself, which implies that $Z_n \in C$.
\end{proof}

\begin{theorem}\label{th:size-of-a}
$|A| \lesssim e^{c m^{2 + \eta(m)} }$ for some sequence $\eta(m) \to 0$ as $m \to \infty$.
\end{theorem}

\begin{proof}
The size of $A$ is at most the number of all lamp configurations with support in $B_{Km}(o)$ times the size of $C$. The number of configurations can be bounded above by $2^{|B_{Km}|}$, which by the growth condition \eqref{eq:growth} is at most $e^{c m^{2 + \varepsilon(m)}}$.

To bound the size of $C$, we use the property that words with small inverted orbits admit a concise description. Every element $\gamma \in \Gamma$ can be described by specifying for each vertex its image under the action of $\gamma$. Now suppose $\gamma$ can be represented by a word $w$ with the property that $d(x,x.w) \leq Km$ for every vertex $x$. Since every vertex $x \in S$ is mapped under the action of $w$ into some other vertex from the ball $B_{Km}(x)$, for a fixed vertex $x$ we have at most $|B_{Km}(x)|$ possible choices.

Now, for a fixed $m$ we have only finitely many types of vertices for which we have to specify their images in order to describe $\gamma$ (since the image of a vertex $x$ under $w$ depends only on the isomorphism type of the ball of radius at most $Km$ around $x$). We distinguish three types of vertices: 1) vertices such that $B_{Km}(x)$ intersects only one level in $S$, 2) $B_{Km}(x)$ intersects two levels in $S$, 3) $B_{Km}(x)$ intersects at least three levels in $S$.

For vertices of the first kind, the ball $B_{Km}(x)$ does not intersect any branching cycle, which means that it looks like a ball in $\Z$ and all vertices of this kind are mapped by $\gamma$ in the same way. Thus we have at most $2 K m$ choices for vertices of this kind.

For vertices of the second kind, each of them must be in a ball of radius $Km$ around a branching point which does not intersect any other branching cycle. Such a ball can have at most $6 Km$ vertices and each of them is mapped into a ball of radius at most $2Km$ around a branching point, which can have at most $cm$ vertices (for some $c$). This give us at most $(cm)^{6Km}$ possibilities.

For vertices of the third kind, we observe that if $B_{Km}(x)$ intersects at least three levels and $x$ belongs to the $k$-th level, then at least one of $\alpha_{k-1}$, $\alpha_{k}$, $\alpha_{k+1}$ is smaller than $2Km$. From this and Assumption \ref{eq:scaling-assumption} it follows that $d(o, x) \leq c m$ for some $c >0$ (like in the proof of Theorem \ref{th:all-x}). Thus we have at most $|B_{cm}(o)|$ vertices of this kind. Since $B_{Km}(x) \subseteq B_{(c+K)m}(o)$, we have at most $|B_{(c+K)m}(o)|$ choices for each vertex. As $\varepsilon(m) \to$, this gives us at most $  |B_{(c+K)m}(o)|^{|B_{cm}(o)|} \leq e^{c m^{2 + o(1)}\log m}$ choices for vertices of this kind.

Thus there at most a constant times $2m \cdot (8m)^{cm} \cdot e^{cm^{2 + o(1)}\log m}$ possible choices determining an element $\gamma$ which can be represented by a word which has small inverted orbits. This gives us $|C| \leq e^{c m^{2 + o(1)} \log m}$ and $|A| \leq e^{c m^{2 + o(1)}} \cdot |C| \leq e^{c m^{2 + o(1)} \log m}$, so the theorem holds with $\eta(m) \lesssim \frac{\log \log m}{\log m}$.
\end{proof}

\begin{corollary}\label{cor:ret-prob}
The return probability for the random walk $\tilde{X}_{n}$ on $G = \Z_2 \wr_{S} \Gamma$ satisfies for all $n \geq 1$:
\[
p_{2n}(e,e) \gtrsim e^{-c n^{1/2 + o(1)} }
\]
\end{corollary}

\begin{proof}

By combining Theorem \ref{th:ret-prob} and Theorem \ref{th:size-of-a} we obtain the bound:
\[
p_{2n}(e,e) \gtrsim e^{-c m^{2 + \eta(m)}} e^{-c \frac{n}{m^2}}
\]
To make this bound optimal we want both terms on the right hand side to be of the same order, which corresponds to taking $m$ such that $\frac{n}{m^2} = m^{2 + \eta(m)}$. This means that $m = n^{1/4 - \varepsilon'(n)}$ for some $\varepsilon'(n) \geq 0$, $\varepsilon'(n) \to 0$. Inserting this back into the lower bound gives us:
\[
p_{2n}(e,e) \gtrsim e^{-c n^{1/2 + f(n)} }
\]
with $f(n) \lesssim \frac{\log \log n}{\log n} = o(1) $ as $n \to \infty$.
\end{proof}

\begin{remark}
One can do a similar calculation for a more general scaling sequence satisfying $\alpha_{n} = \alpha^{n + o(n)}$, with $\alpha > 1$, which then gives:
\[
|A| \lesssim e^{c m^{d + o(1)} }
\]
and
\[
p_{2n}(e,e) \gtrsim e^{-c n^{\frac{d}{d+2}} }
\]
with $d = 1 + \frac{\log 2}{\log \alpha} + o(1)$ as $m \to \infty$.
\end{remark}

We can now prove the main theorem:

\begin{proof}[Proof of Theorem \ref{th:main}]
Take $G = \Z_2 \wr_{S} \Gamma$ for $S$ and $\Gamma$ as above. By Corollary \ref{cor:ret-prob} the return probability for the switch-walk-switch random walk $\mu$ on $G$, induced from the simple random walk on $\Gamma$ and a uniform distribution on $\Z_2$, satisfies:
\[
p_{2n}(e,e) \gtrsim e^{-c n^{1/2 + o(1)} }
\]
which gives the return probability exponent $\overline{\gamma} \leq 1/2$. The induced random walk on $S$ is the simple random walk, which by Proposition \ref{proposition:transience} is transient, so by Theorem \ref{th:liouville} the group $G$ supports nonconstant bounded harmonic functions. Thus $G$ has both $\overline{\gamma} \leq 1/2$ and the non-Liouville property.
\end{proof}

\section*{Acknowledgements}\label{acknowledgements}

We would like to thank Gady Kozma for bringing our attention to groups permuting vertices of slowly growing trees. We would also like to thank anonymous referees for their valuable comments.

\bibliography{bibliography}{}
\bibliographystyle{amsalpha}

\bigskip
\noindent
 Michał Kotowski \\
 Department of Mathematics, University of Toronto \\
 Bahen Centre ,40 St. George St., Toronto, Ontario \\
 CANADA M5S 2E4
 \\
\noindent
{E-mail:} {\tt michal@math.toronto.edu} \\
\noindent
\url{http://www.math.toronto.edu/~michal/}

\bigskip
\noindent
 B\'{a}lint Virag \\
 Department of Mathematics, University of Toronto \\
 Bahen Centre ,40 St. George St., Toronto, Ontario \\
 CANADA M5S 2E4
 \\
\noindent
{E-mail:} {\tt balint@math.toronto.edu} \\
\noindent
\url{http://www.math.toronto.edu/~balint/}

\end{document}